\documentclass[a4paper,11pt,english]{article}
\usepackage{amsmath}
\usepackage{amsfonts}
\usepackage{amssymb}
\usepackage{amsthm}
\usepackage[all]{xy}
\usepackage{geometry}
\usepackage{babel}
\usepackage{graphicx}
\usepackage{multirow}

\newtheorem {thm}{Theorem}
\newtheorem* {thm*}{Theorem}

\newtheorem* {cor*}{Corollary}
\newtheorem {lem}[thm]{Lemma}

\newtheorem {defi}[thm]{Definition}
\newtheorem {rem}[thm]{Remark}

\theoremstyle{definition}

\newtheorem* {conj*}{Conjecture}

\newtheorem* {defi*}{Definition}

\DeclareMathOperator{\ord}{ord}
\DeclareMathOperator{\Gal}{Gal}
\DeclareMathOperator{\Hom}{Hom}
\DeclareMathOperator{\lcm}{lcm}

\DeclareMathOperator{\dens}{dens}

\newcommand{\Q}{\mathbb{Q}}

\renewcommand{\a}{\alpha}
\renewcommand{\b}{\beta}

\newcommand{\g}{\gamma}

\newcommand{\p}{\mathfrak{p}}

\newcommand{\e}{\emph}

\author{Antonella Perucca}
\title{The order of the reductions of an algebraic integer}
\date{ }

\begin{document}
\maketitle

\begin{abstract}
Let $K$ be a number field, and let $a\in K^\times$. Fix some prime number $\ell$.
We compute the density of the following set: the primes $\mathfrak p$ of $K$ such that the multiplicative order of the reduction of $a$ modulo $\p$ is coprime to $\ell$ (or, more generally, has some prescribed $\ell$-adic valuation). 
We evaluate the degree over $K$ of extensions of the form  $K(\zeta_{\ell^m}, \sqrt[\ell^n]{a})$ with $n\leq m$, which are obtained by adjoining a root of unity of order $\ell^m$ and the $\ell^n$-th roots of $a$, as this is needed for computing the above density.
\end{abstract}

\section{Introduction}

Let $K$ be a number field, and let $a\in K^\times$. Fix some prime number $\ell$. The problem that we solve in this paper is the following: computing the density of the set of primes $\p$ of $K$ such that the reduction of $a$ modulo $\p$ has multiplicative order coprime to $\ell$ (or, more generally, has some prescribed $\ell$-adic valuation). This is a natural density and its value may be expressed in terms of the degrees of some cyclotomic and cyclic Kummer extensions:
\begin{equation*}
D_{\ell}(a)= \sum_{n\geq 0}\;\Big(\; \frac{1}{[K(\zeta_{\ell^n}, \sqrt[\ell^{n}]{a}):K]}-\frac{1}{[K(\zeta_{\ell^{n+1}},\sqrt[\ell^{n}]{a}):K]}\; \Big)
\end{equation*}
where $\zeta_{\ell^m}$ is a root of unity of order $\ell^m$ and $\sqrt[\ell^{n}]{a}$ is  an $\ell^n$-th root of $a$.
We evaluate this sum and provide formulas which depend only on some meaningful parameters. We may suppose that $a$ is not a root of unity, because in this case the density is trivially either $1$ or $0$ according to whether the order of $a$ is coprime to $\ell$ or not. In particular we prove:

\begin{thm}\label{intro1}
If $\zeta_{\ell}\notin K$ and if $\zeta_{\ell^2}\notin K(\zeta_{\ell})$,  then we have 
\begin{equation*}
D_{\ell}(a)=1-\frac{1}{[K(\zeta_{\ell}):K]}\cdot \frac{\ell^{1-d}}{\ell+1}
\end{equation*}
where $d\geq 0$ is the greatest integer such that $a$ is an $\ell^d$-power in $K$. 
\end{thm}

\begin{thm}\label{intro2}
If $\ell$ is odd and $\zeta_{\ell}\in K$, or if $\ell=2$ and $\zeta_{4}\in K$ then we have
$$D_{\ell}(a)=1-\frac{1}{\ell^{n}\cdot (\ell+1)} \qquad \text{or}\qquad D_{\ell}(a)=\frac{1}{\ell^{n}\cdot (\ell+1)}$$
for some integer $n\geq 0$.
\end{thm}

If $\ell$ is odd or if $\zeta_4\in K$ the cyclotomic extension $K(\zeta_{\ell^m})$ over $K$ is cyclic for every $m\geq 1$. In this case the formulas that we get are extremely simple: about the field $K$, we only need to know the degree over $\mathbb Q$ of the intersection of $K$ with the $\ell$-th cyclotomic tower; about the element $a$ what matters are only two integer numbers expressing the divisibility of $a$ in $K^\times$. 
The key property is the following: for $\ell$ odd or if $\zeta_4\in K$, we have $$[K(\zeta_{\ell^m}, \sqrt[\ell^{n}]{a}):K(\zeta_{\ell^m})]=\ell^n\qquad \textrm{ for every $m,n$ with $n\leq m$}$$ provided that $a$ is strongly indivisible, by which we mean that $a\xi$ has no $\ell$-th roots in $K$ for every root of unity $\xi\in K$. If $\ell=2$ and $\zeta_4\notin K$ there is one special case to consider, because it may happen that the square roots of $a$ are contained in $K(\zeta_{2^m})$.

The history of the problem addressed in this paper goes back to Hasse in the 1960's, who worked with $K=\mathbb Q$, see  \cite{Hasse1, Hasse2}. Our general approach gives also insight on the many cases which were known to occur for $K=\mathbb Q$ and $\ell=2$, for example why $D_2(a)=7/24$ for $a=\pm 2$ while $D_2(4)=7/12$ and $D_2(-4)=1/3$. In particular there has been a lot of work in the case $K=\mathbb Q$ and $\ell=2$. Given a sequence of the form $c^k +d^k$ with $c$ and $d$ integers, one can ask for the rational primes $p$ such that $p$ divides $c^k +d^k$ for some $k$: this is the same as asking for the order of $\frac{c}{d}$ to be even modulo $p$. There are also generalizations to Lucas sequences, this corresponds to taking $K$ to be quadratic. For this kind of questions, we refer to the survey of Ballot~\cite{Ballot} and to the book by Everest and others~\cite{Everest}.
It is important to notice that Jones and Rouse \cite{Jones_Rouse} were able to solve our problem  (in the generic case) in the more general context of commutative algebraic groups, by using techniques from dynamical systems. We also signal the generalization where $\ell$ is replaced by some square-free integer: this question has been addressed by Wiertelak, Pappalardi and Moree if $K=\mathbb Q$, see \cite{Wiertelak, Pappalardi, Moree_divisibility}; the case of a  number field  will be included in \cite{Debry}. 
Another generalization is considering the primes $\mathfrak p$ of $K$ such that the reduction of $a$ modulo $\p$ has multiplicative order in some arithmetic progression, see the papers by Moree and by Ziegler~\cite{Moree_residue, Ziegler}.
For an extensive description of related problems and numerous additional references, we refer to the survey of Moree~\cite{MoreeArtin}, and in particular sections~9.2~to~9.4. 

The structure of the paper is as follows: Section~\ref{Kummersec} contains preliminaries on Kummer extensions. In section~\ref{divprop} we find a convenient way of expressing the divisibility properties of $a$ and determining the parameters that will appear in the subsequent formulas. In section~\ref{secdeg} we evaluate the degree of $K(\zeta_{\ell^m}, \sqrt[\ell^{n}]{a})$ over $K(\zeta_{\ell^m})$ for $n\leq m$, where $K$ is any field and $\ell$ is coprime to the characteristic. In section~\ref{formula1} we restrict to $K$ a number field: we give the formulas for the density and provide several examples, see tables~\ref{Q3},~\ref{Q2},~\ref{Q22}.

It is a pleasure to thank Hendrik Lenstra for essential contributions to this paper, and Chantal David, Peter Jossen, Falko Lorenz, Pieter Moree, Andrzej Schinzel, Jean-Pierre Serre and the referee for useful remarks and comments.

\subsection*{Notation}

In this paper, $K$ is a field and $\ell$ is a prime number different from the characteristic of $K$: only in section~\ref{formula1} we suppose that $K$ is a number field. For $m\geq 1$, we write $\zeta_{\ell^m}$ for a root of unity of order $\ell^m$ and $\mu_{\ell^m}$ for the group of $\ell^m$-th roots of unity. Moreover, we write $K_{\ell^m}:=K(\zeta_{\ell^m})$ and $K_{\ell^\infty}:=\cup_{m\geq 1} K_{\ell^m}$. We denote by $K^\times$ the multiplicative group of $K$.

\section{Cyclic Kummer extensions}\label{Kummersec}

Suppose that $\a$ is algebraic over $K$, and that there is $n\geq 1$ such that the following holds:  $\a^{\ell^n}=a$ for some $a\in K^\times$ and the $\ell^n$-th roots of unity are contained in $K$. We then write $K(\a)=K(\sqrt[\ell^n]{a})$.  The extension $K(\a)/K$ is obtained by adjoining one (or equivalently, all) $\ell^n$-th roots of $a$. It is a Galois extension of $K$, being the splitting field of $X^{\ell^n}-a$. The Galois group is cyclic of order dividing $\ell^n$. The extension $K(\a)/K$ is called a cyclic Kummer extension. We now state an important result, from which the considerations below easily follow (see \cite[ch.~VI, \S6~and~\S8]{Lang}, \cite[\S2]{Birch}, \cite[ch.~IV, \S3]{Neukirch}):

\begin{thm}\label{KTHM}
If $n\geq 1$ and $\zeta_{\ell^n}\in K$, then the following holds:
\begin{enumerate}
\item A finite Galois extension of $K$ whose Galois group is cyclic of order $\ell^n$ is a cyclic Kummer extension of the form $K(\sqrt[\ell^n]{a})$ for some $a\in K^\times$.
\item If $a, b\in K^{\times}$ are such that $K(\sqrt[\ell^n]{a})=K(\sqrt[\ell^n]{b})$ has degree $\ell^n$ over $K$, then we have $a=b^t \g^{\ell^n}$ for some $\g\in K^\times$ and for some integer $t$ coprime to $\ell$.
\end{enumerate}
\end{thm}

Let $L=K(\sqrt[\ell^n]{a})$ be a cyclic Kummer extension of $K$ of degree $\ell^n$. If $m\leq n$, the field $K(\sqrt[\ell^n]{a^{\ell^{(n-m)} }})=K(\sqrt[\ell^m]{a})$ is the unique subextension of $K(\sqrt[\ell^n]{a})$ of degree $\ell^m$, and it is a cyclic Kummer extension. On the other hand if $t$ is coprime to $\ell$, we have $K(\sqrt[\ell^n]{a^{t}})=K(\sqrt[\ell^n]{a})$.

Now allow the degree of $L/K$ to be a proper divisor of $\ell^n$. Consider the subgroup of $K^{\times}$ consisting of the elements which have some (hence all) $\ell^n$-th roots in $L$. This subgroup contains $K^{\times \ell^n}$ so in fact we may associate to $L$ a subgroup $\Delta$ of $K^{\times}/K^{\times \ell^n}$.
The group $\Delta$ is cyclic and it is generated by the class of $a$. Note, the function $L\mapsto \Delta$ maps cyclic Kummer extensions of $K$ of degree $\ell^d$ to cyclic subgroups of $K^{\times}/K^{\times \ell^n}$ of order $\ell^d$, and it is a bijection. 
Moreover, there is a group isomorphism
$$\Delta\simeq \Hom(\Gal(L/K), \mu_{\ell^n})$$
such that the class of $a$ is mapped to the character $\chi_{a}:\sigma\mapsto \a^{\sigma}\cdot\alpha^{-1}$, where $\a$ is such that $\a^{\ell^n}=a$ (this character does not depend on the choice of $\a$).

The degree $[L:K]$ divides $\ell^n$ and it can be equivalently characterized as follows: the order of $\chi_a$; the order of the class of $a$ in $K^{\times}/K^{\times \ell^n}$; the smallest integer $\ell^d$ such that $\a^{\ell^d}\in K^{\times}$ for some $\a$ satisfying $\a^{\ell^n}=a$ (this does not depend on  the choice of $\a$); the smallest integer $\ell^d$ such that $a$ has some $\ell^{(n-d)}$-th roots in $K$. 

\begin{lem}\label{KTdifferent}\label{KTproduct}
Let $a, b\in K^{\times}$. Let $n\geq 1$ and $\zeta_{\ell^n}\in K$.  If the fields $K(\sqrt[\ell^n]{a})$ and $K(\sqrt[\ell^n]{b})$ are linearly disjoint over $K$ or the $\ell$-adic valuation of their degrees over $K$ are different, then we have
$$\big[K(\sqrt[\ell^n]{ab}): K\big]=\lcm \big( [K(\sqrt[\ell^n]{a}): K], [K (\sqrt[\ell^n]{b}): K] \big)\,.$$
\end{lem}
\begin{proof}
Write $K_a:=K(\sqrt[\ell^n]{a})$, $K_b:=K(\sqrt[\ell^n]{b})$, $K_{ab}:=K(\sqrt[\ell^n]{ab})$, and call respectively $d_a$, $d_b$ and $d_{ab}$ the order of $\chi_{a}$, $\chi_{b}$ and $\chi_{ab}$. To prove that $d_{ab}$ divides $\lcm(d_a, d_b)$, it suffices to notice that for every $\sigma\in \Gal(\bar{K}/K)$ we have
$$\chi_{ab}(\sigma|_{K_{ab}})=\chi_{a}(\sigma|_{K_{a}})\cdot \chi_{b}(\sigma|_{K_{b}})\,.$$
We are left to show that for the $\ell$-adic valuation $v_{\ell}(d_{ab})\geq \max\{v_{\ell}(d_{a}), v_{\ell}(d_{b})\}$ holds. Without loss of generality suppose $v_{\ell}(d_{a})\geq v_{\ell}(d_{b})$. By assumption $K_a$ and $K_b$ are linearly disjoint over $K$ or $v_{\ell}(d_{a})> v_{\ell}(d_{b})$. Then we can find $\sigma\in \Gal(\bar{K}/K)$ such that the order of $\chi_{a}(\sigma|_{K_{a}})$ has $\ell$-adic valuation $v_{\ell}(d_a)$ while the order of $\chi_{b}(\sigma|_{K_{b}})$ has  $\ell$-adic valuation respectively $0$ or at most $v_{\ell}(d_{b})$. In both cases, the order of $\chi_{ab}(\sigma|_{K_{ab}})$ has $\ell$-adic valuation $v_{\ell}(d_{a})$. \end{proof}

\begin{lem}\label{KTtower}
Let $a\in K^{\times}$. Let $n\geq 1$ and $\zeta_{\ell^n}\in K$.  Either $K(\sqrt[\ell^n]{a})= K$ holds, or there is some smallest integer $h$ with $1\leq h\leq n$ such that $K(\sqrt[\ell^h]{a})\neq K$ and we have 
$$[K(\sqrt[\ell^n]{a}): K]=\ell^{n-h+1}\,.$$
\end{lem}
\begin{proof}
By assumption, $\chi_{a^{\ell^{(n-h+1)}}}$ is the trivial character while $\chi_{a^{\ell^{(n-h)}}}$ is non-trivial. We deduce that the second has order $\ell$ hence $\chi_a$ has order $\ell^{n-h+1}$.
\end{proof}

We will apply in several occasions the following theorem of Schinzel:

\begin{thm}[{\cite[thm.~2]{Schinzel_abelianbinomials}}]\label{Schinzel}
Let $a\in K^{\times}$. For $n\geq 1$, the extension $K(\zeta_{\ell^n}, \sqrt[\ell^n]{a})/K$ is abelian if and only if $a^{\ell^m}=\gamma^{\ell^n}$ for some $\gamma\in K^{\times}$ and for some $m\leq n$ such that $\zeta_{\ell^m}\in K$.
\end{thm}

Note, a shorter proof for Schinzel's result was given independently by Stevenhagen and by W\'ojcik \cite{Lenstra_H, Wojcik}.

\section{Divisibility properties}\label{divprop}

In this section we investigate the divisibility properties of the elements of $K$:

\begin{defi}\label{l_strongly_indivisible} 
Let $a\in K^\times$. We say that $a$ is \emph{divisible} (in $K$) if $a$ has some $\ell^n$-th root in $K$ for every $n\geq 1$. We say that $a$ is \emph{strongly indivisible} (in $K$) if $a\xi$ has no $\ell$-th roots in $K$ for every root of unity $\xi\in K\cap \mu_{\ell^\infty}$, or equivalently for every root of unity $\xi\in K$. 
\end{defi}

 The roots of unity of order coprime to $\ell$ are divisible. The roots of unity of order $\ell^m$ for some $m>0$ are divisible if and only if $K=K_{\ell^\infty}$, and they are never strongly indivisible.
 
\begin{lem}~\label{noSTARell} Let $a\in K^\times$. 
\begin{enumerate}
\item If $K\neq K_{\ell}$ or $K=K_{\ell^\infty}$, then either $a$ is divisible in $K$ or it is of the form $a=b^{\ell^d}$ for some $d\geq 0$ and for some $b\in K$ strongly indivisible. The integer $d$ does not depend on the choice of $b$.

\item If $K=K_{\ell}\neq K_{\ell^\infty}$, let $t>0$ be the greatest integer such that $K=K_{\ell^t}$. Then exactly one of the following holds:
\begin{enumerate}
\item[(i)] $a=b^{\ell^d}$ for some $d\geq 0$ and for some $b\in K$ strongly indivisible;
\item[(ii)] $a=b^{\ell^d}\xi $ for some $d> 0$, for some $b\in K$ strongly indivisible and for some root of unity $\xi\in K$ of order $\ell^r$ with $r>\max(0, t-d)$;
\item[(iii)] $a\xi$ is divisible in $K$ for some uniquely determined root of unity $\xi\in K\cap \mu_{\ell^\infty}$.
\end{enumerate}
The integers $d$ and $r$ do not depend on the choice of $b$.

\end{enumerate}
\end{lem}
\begin{proof}[Proof: Proof of 1.] Either $a$ is divisible or there is some maximal $d\geq 0$ such that we can write $a=b^{\ell^d}$ for some $b\in K$. Then $b$ is strongly indivisible by the maximality of $d$. If we have $a=\b^{\ell^\delta}$ with $\b\in K$ and $\delta<d$ we can easily see that $\b$ is not strongly indivisible. 

\emph{Proof of 2.} Suppose first that there is some maximal $d\geq 0$ such that we can write $a=b^{\ell^d}\xi$ for some $b\in K$ and for some root of unity $\xi\in K\cap \mu_{\ell^\infty}$. Then $b$ is strongly indivisible by the maximality of $d$. If the order of $\xi$ divides $\ell^{t-d}$  then $\xi$ has $\ell^d$-roots in $K$ so up to replacing $b$ we may suppose $\xi=1$.
If $a=\beta^{\ell^\delta}\zeta$ for some $\beta\in K$ strongly indivisible and some root of unity $\zeta\in K\cap \mu_{\ell^\infty}$ we cannot have $\delta<d$ hence $\delta=d$. Having $a=\beta^{\ell^d}\zeta=b^{\ell^d}\xi$ implies that $\xi\zeta^{-1}$ has $\ell^{d}$-th roots in $K$. Under the assumption that $\xi, \zeta$ are either $1$ or have order greater than $\max(0, t-d)$ we deduce that $\xi, \zeta$ have the same order.

Now suppose that there is no such maximal $d$ as above. Since the number of roots of unity in $K\cap \mu_{\ell^\infty}$ is finite we may find infinitely many integers $h\geq 1$ such that $a=b_h^{\ell^h}\zeta$ for some $b_h\in K^\times$ and for some fixed $\zeta\in K\cap \mu_{\ell^\infty}$. Then $a\xi$ where $\xi=\zeta^{-1}$ is divisible in $K$.
Since $K\neq K_{\ell^\infty}$ the root of unity $\xi$ is uniquely determined. \end{proof}

\section{Evaluating the degree of Kummer extensions}\label{secdeg}

We want to evaluate the degree $[K_{\ell^m}(\sqrt[\ell^n]{a}):K_{\ell^m}]$ for $m\geq n>0$ and $a\in K^\times$. If $a$ is divisible in $K$ then the above Kummer extension is trivial. If $K=K_{\ell}\neq K_{\ell^\infty}$ it may happen that $a$ is not divisible in $K$ but $a\xi$ is divisible in $K$ for some root of unity $\xi\in K$: in this case we may replace $a$ by $\xi^{-1}$ and apply lemma~\ref{cyclotomic}. The remaining cases are discussed in theorems~\ref{odd_or_i}~and~\ref{case_inotK}: in particular, $a$ is not a root of unity.

\subsection{Cyclotomic extensions}

Since cyclotomic extensions are extremely well-known (see for example \cite{Birch, Lang, Neukirch, Washington}) the proof of the following two lemmas is left as an exercise for the reader. Note, the extension $\Gal(K_{\ell^m}/K)$ may fail to be cyclic only in characteristic zero and if $\ell=2$ and $\zeta_4\notin K$:

\begin{lem}\label{cyclotomic2}
For every $n\geq 2$ the fields
\begin{equation*}\label{++}
\Q_{2^{n}}^+:=\Q\big(\zeta_{2^n}+\zeta_{2^n}^{-1}\big) =\Q\Big(\sqrt{\zeta_{2^{n-1}}+\zeta_{2^{n-1}}^{-1}+2}\Big)
\end{equation*}
\begin{equation*}\label{-}
\Q_{2^{n}}^-:=\Q\big(\zeta_4\cdot(\zeta_{2^n}+\zeta_{2^n}^{-1})\big)=\Q\Big(\sqrt{-(\zeta_{2^{n-1}}+\zeta_{2^{n-1}}^{-1}+2)}\Big)
\end{equation*}
 do not contain $\zeta_4$. For every $n\geq 2$ the following diagram shows all subfields of $\Q_{2^{n+2}}$ containing $\Q_{2^{n}}^+$: every  field inclusion is marked as a composition of arrows, and each arrow corresponds to an extension of degree $2$. 

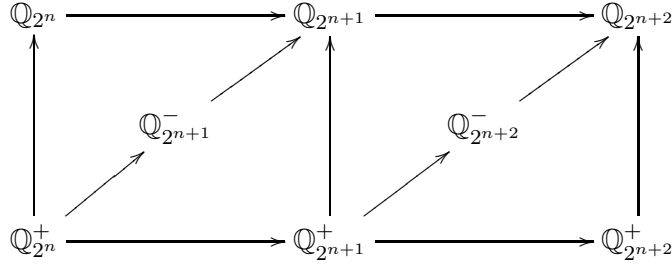
\begin{center}
$\xymatrix{ 
\Q_{2^{n}} \ar[rr] & &\Q_{2^{n+1}}\ar[rr] & & \Q_{2^{n+2}} \\ 
 &  \Q_{2^{n+1}}^- \ar [ur]  &  &  \Q_{2^{n+2}}^- \ar [ur]  \\
\Q_{2^{n}}^+ \ar[uu] \ar [ur] \ar[rr] & &\Q_{2^{n+1}}^+\ar[uu]\ar[rr]\ar[ur]  & & \Q_{2^{n+2}}^+  \ar[uu] } $
\end{center}
\end{lem}

We will also use the notation $\Q_{2^{\infty}}^+:=\cup_{n\geq 2} \Q_{2^{n}}^+$.

\begin{lem}\label{cyclotomic} Let $K\neq K_{\ell^\infty}$. 
\begin{enumerate}
\item If $\ell$ is odd, then the degree $[K_{\ell}:K]$ divides $\ell-1$. If $K_{\ell}\neq K_{\ell^\infty}$, then for every $m\geq 2$ we have
\begin{equation*}
[K_{\ell^m}:K]=[K_{\ell}:K]\cdot  \ell^{\max(0,m-t)}
\end{equation*}
where $t\geq 1$ is the greatest integer such that $K_{\ell}=K_{\ell^t}$. If $K_{\ell}= K_{\ell^\infty}$, we have $[K_{\ell^m}:K]=[K_{\ell}:K]$ for every $m\geq 2$\,.
\item If $\ell=2$, we have $K=K_2$ and the degree $[K_4:K]$ is $1$ if $\zeta_4\in K$ and it is $2$ otherwise.
If $K_{4}\neq K_{2^\infty}$, then for every $m\geq 2$ we have
\begin{equation*}
[K_{2^m}:K]=[K_{4}:K]\cdot  2^{\max(0,m-t)}
\end{equation*}
where $t\geq 2$ is the greatest integer such that $K_{4}=K_{2^t}$. We have $K_4=K_{2^\infty}$ if and only if $K$ has characteristic zero and $K\cap \Q_{2^\infty}=\Q_{2^\infty}^+$. In this case, we have $[K_{2^m}:K]=2$ for every $m\geq 2$.
\end{enumerate}
\end{lem}

\subsection{The case $\ell$ odd or $\zeta_4\in K$}

\begin{thm}\label{odd_or_i}
Let $\ell$ be odd, or let $\ell=2$ and $\zeta_4\in K$. Let $a\in K^\times$, and let $m\geq n> 0$.
\begin{enumerate}
\item If $a=b^{\ell^d}$ for some $b\in K^\times$ strongly indivisible and for some $d\geq 0$, then we have
$$[K_{\ell^m}(\sqrt[\ell^n]{a}): K_{\ell^m}]=\ell^{\max(0,n-d)}\,.$$
\item If $K=K_{\ell}\neq K_{\ell^\infty}$, let $t\geq 1$ be the greatest integer such that $K_{\ell^t}=K$ (for $\ell=2$ we are assuming $t\geq 2$). Without loss of generality suppose $m\geq t$. If $a=b^{\ell^d}\xi$ for some $b\in K^\times$ strongly indivisible, for some $d> 0$ and for some root of unity $\xi\in K$ of order $\ell^r$ with  $r>\max(0, t-d)$, then we have
$$[K_{\ell^m}(\sqrt[\ell^n]{a}): K_{\ell^m}]=\ell^{\max(0,{n-d}, {n+r-m})}\;.$$
\end{enumerate}
\end{thm}
\begin{proof}[Proof: Proof of 1.]  We have $K_{\ell^m}(\sqrt[\ell^n]{a})=K_{\ell^m}(\sqrt[\ell^{n-d}]{b})$ if $d\leq n$ and $K_{\ell^m}(\sqrt[\ell^n]{a})=K_{\ell^m}$ if $d\geq n$. Thus by lemma~\ref{KTtower} it suffices to prove that $[K_{\ell^m}(\sqrt[\ell]{b}): K_{\ell^m}]=\ell$.
This degree divides $\ell$ so suppose that it is $1$. In particular $K_{\ell}(\sqrt[\ell]{b})$ is an abelian extension of $K$. If $K\neq K_{\ell}$ then by theorem \ref{Schinzel} we deduce $b=\gamma^{\ell}$ for some $\gamma\in K$, a contradiction. We may easily deduce the same if  $K=K_{\ell^{\infty}}$. We are left to deal with the following case: there is some greatest integer $t>0$ such that $K=K_{\ell^t}$. Every finite subextension of $K_{\ell^\infty}/K$ is cyclic because $\zeta_4\in K$ if $\ell=2$.
Since $K\neq K(\sqrt[\ell]{b})\subseteq K_{\ell^\infty}$ we must have $K(\sqrt[\ell]{b})=K_{\ell^{t+1}}$.  Thus $K(\sqrt[\ell]{b})=K(\sqrt[\ell]{\zeta})$ for some root of unity $\zeta\in K$  of order $\ell^t$. By theorem~\ref{KTHM} we deduce $b=\gamma^{\ell}\zeta$, a contradiction.

\emph{Proof of 2.} Let $\g=a\xi^{-1}=b^{\ell^d}$. By the previous case, the extensions $K_{\ell^m}(\sqrt[\ell^{n}]{\g})$ and $K_{\ell^m}(\sqrt[\ell^{n}]{\xi})\subseteq K_{\ell^{\infty}}$ are linearly disjoint over $K_{\ell^m}$. Then by lemma~\ref{KTproduct} the requested degree is the least common multiple of the degrees of $K_{\ell^m}(\sqrt[\ell^{n}]{\g})$ and of $K_{\ell^m}(\sqrt[\ell^{n}]{\xi})$ over  $K_{\ell^m}$. These two degree were evaluated respectively in the previous case and in lemma~\ref{cyclotomic}.
\end{proof}

\subsection{The case $\ell=2$ and $\zeta_4\notin K$}

\begin{lem}\label{lemmalemma}\label{squarespecial}
If $\zeta_4\notin K$ and $a\in K^\times$ is strongly indivisible, then we have
$$[K_{2^{\infty}}(\sqrt{a}): K_{2^{\infty}}]=2$$
unless $K$ has characteristic zero  and the conditions
\begin{equation}\label{exceptional1}
\left\{ \begin{array}{l}
\vspace{0.2cm}
K\cap \Q_{2^{\infty}}=\Q(\zeta_{2^s}+\zeta^{-1}_{2^s}) \\
a=\pm(\zeta_{2^s}+\zeta_{2^s}^{-1}+2)\cdot \b^2 \\
\end{array}\right.
\end{equation}
hold for some $\b\in K^\times$ and for some $s\geq 2$. In this case, we have $K(\sqrt{a})\not\subseteq K_{2^{s}}$ and $K(\sqrt{a})\subseteq K_{2^{s+1}}\subseteq K_{2^{\infty}}$ and in particular we have $K_4\neq K_{2^\infty}$.
\end{lem}
\begin{proof}
If $K_4(\sqrt[4]{a})\subseteq K_{2^{\infty}}$ then $K_4(\sqrt[4]{a})$ is an abelian extension of $K$ so by theorem \ref{Schinzel} we have $a^{2}=\gamma^{4}$ for some $\gamma\in K$. This implies $a=\pm\gamma^{2}$, contradicting that $a$ is strongly indivisible. Note, $K(\sqrt{a})$ does not contain $\zeta_4$ because otherwise $K(\sqrt{a})=K(\sqrt{-1})$ and hence we would have $a=-\gamma^2$ for some $\gamma\in K$, a contradiction.
Also $K(\sqrt{a})\neq K$, again because $a$ is strongly indivisible.

If $K_{2^\infty}/K$ is cyclic then from $K_4\neq K(\sqrt{a})$ we deduce that $K(\sqrt{a})\not\subseteq K_{2^{\infty}}$. From now on assume that $K_{2^\infty}/K$ is not cyclic so in particular $K$ has characteristic zero and $K\cap \Q_{2^{\infty}}\subseteq \Q_{2^{\infty}}^+$.
Suppose that $K(\sqrt{a})\subseteq K_{2^{\infty}}$. Since $[K(\sqrt{a}): K]=2$ and $\zeta_4\notin K(\sqrt{a})$, by lemma~\ref{cyclotomic2} we have $K\cap \Q_{2^{\infty}}\neq \Q_{2^{\infty}}^+$ hence there is some integer  $s\geq 2$ satisfying $K\cap \Q_{2^{\infty}}=\Q(\zeta_{2^s}+\zeta^{-1}_{2^{s}})$. Moreover  we have 
$$K(\sqrt{a})=K\Big(\sqrt{\pm(\zeta_{2^{s}}+\zeta^{-1}_{2^{s}}+2)}\Big)\,.$$
It follows that $a=\pm(\zeta_{2^{s}}+\zeta^{-1}_{2^{s}}+2)\cdot\b^2$ for some $\b\in K^\times$. By lemma~\ref{cyclotomic2} we have $K(\sqrt{a})\subseteq K_{2^{s+1}}$ and $K(\sqrt{a})\not\subseteq K_{2^s}$. In particular we have $K(\sqrt{a})\subseteq K_{2^\infty}$ and $K_4\neq K_{2^\infty}$. 
\end{proof}

\begin{thm}\label{case_inotK}
Let $\zeta_4\notin K$, let $a\in K^\times$ and let $m\geq n> 0$. If $K_4\neq K_{2^\infty}$, let $s$ be the greatest integer such that $K_4=K_{2^{s}}$.
\begin{enumerate}
\item If $a=b^{2^d}$ for some $d\geq 0$ and for some $b\in K^{\times}$ strongly indivisible, then we have
$$[K_{2^m}(\sqrt[2^n]{a}): K_{2^m}]=2^{\max(0,n-d)}$$
unless we are in the case $K_4\neq K_{2^\infty}$ and $K(\sqrt{b})\subseteq K_{2^{\infty}}$ and  $m\geq s+1$, in which we have
$$[K_{2^m}(\sqrt[2^n]{a}): K_{2^m}]=2^{\max(0,n-d-1)}\,.$$
\item If $a=-b^{2^d}$ for some $d>0$ and for some $b\in K^{\times}$ strongly indivisible, then we have
$$[K_{2^m}(\sqrt[2^n]{a}): K_{2^m}]=[K_{2^m}(\sqrt[2^n]{-a}): K_{2^m}]$$
with the following exceptions:
\begin{itemize}
\item If $m=n=1$ we have $[K(\sqrt{-a}): K]=1$ and $[K(\sqrt{a}): K]=2$.
\item If $K_4\neq K_{2^\infty}$ and $m=n\geq s$ and $[K_{2^m}(\sqrt[2^n]{-a}): K_{2^m}]=1$ then we have $[K_{2^m}(\sqrt[2^n]{a}): K_{2^m}]=2$.
\item If $K_4\neq K_{2^\infty}$, $K(\sqrt{b})\subseteq K_{2^{\infty}}$, $m=n=s=d+1$ 
and $[K_{2^m}(\sqrt[2^n]{-a}): K_{2^m}]=2$ then we have $[K_{2^m}(\sqrt[2^n]{a}): K_{2^m}]=1$.
\end{itemize}
\end{enumerate}
\end{thm}

\begin{proof}[Proof: Proof of 1.] By combining lemma~\ref{KTtower} and lemma~\ref{squarespecial} we have 
\begin{equation*}
[K_{2^m}(\sqrt[2^n]{b}): K_{2^m}]= \left\{
\begin{array}{ll}
\vspace{0.2cm}
2^{n-1} &  \text{if $K(\sqrt{b})\subseteq K_{2^{\infty}}$ and $m\geq s+1$}\\
2^n &  \text{otherwise}\\
\end{array}
\right.
\end{equation*}
Then it suffices to notice that $K_{2^m}(\sqrt[2^n]{a})=K_{2^m}(\sqrt[2^{n-d}]{b})$ if $n>d$ and $K_{2^m}(\sqrt[2^n]{a})=K_{2^m}$ if $n\leq d$.

\emph{Proof of 2.} Let $h\geq 0$ be such that $2^h=[K_{2^m}(\sqrt[2^n]{-a}): K_{2^m}]$.
 If $m=1$ then $n=1$ so we have $K_{2}(\sqrt{-a})=K_{2}$ and the degree of $K_{2}(\sqrt{a})=K_{4}$ over $K_2$ is $2$. Now let $m\geq 2$. By lemma~\ref{cyclotomic} we have $K_{2^m}(\sqrt[2^n]{-1})\subseteq K_{2^m}$ unless $K_4\neq K_{2^\infty}$ and $n=m\geq s$. Thus apart from this case we know that 
$K_{2^m}(\sqrt[2^n]{a})=K_{2^m}(\sqrt[2^n]{-a})$ and the requested degree is $2^h$.
Suppose from now on that $K_4\neq K_{2^\infty}$ and $n=m\geq s$.

By  lemma~\ref{cyclotomic} we have $[K_{2^m}(\sqrt[2^n]{-1}):K_{2^m}]=2$. If $h=0$ the requested degree is $2$ because $K_{2^n}(\sqrt[2^n]{a})=K_{2^n}(\sqrt[2^n]{-1})$. If $h>1$ the requested degree is $2^h$ by lemma~\ref{KTdifferent}.
We are left with the case $h=1$, in which the two fields $K_{2^n}(\sqrt[2^n]{-a})$ and $K_{2^n}(\sqrt[2^n]{-1})=K_{2^{n+1}}$ have both degree $2$ over $K_{2^n}$. If these two fields are different then they are linearly disjoint and the requested degree is $2$ by lemma~\ref{KTproduct}. If the two fields are equal then $a=(-a)(-1)$ is a $2^n$-th power in $K_{2^n}$ and the requested degree is $1$.

We are left to prove that the two fields $K_{2^n}(\sqrt[2^n]{-a})$ and $K_{2^{n+1}}$ coincide if and only if $K(\sqrt{b})\subseteq K_{2^{\infty}}$ and $n=s=d+1$. Suppose that the two fields coincide. Since $K_{2^{n+1}}\neq K_{2^{n}}$ we have $n>d$ and we are assuming $K_{2^n}(\sqrt[2^{n-d}]{b})=K_{2^{n+1}}$. Then we are in the special case described in \eqref{exceptional1} and $s+1\geq n+1$ must hold. 
By lemma~\ref{KTtower} we have $n-d=1$ because $n=s$ and \eqref{exceptional1} imply $K_{2^n}(\sqrt{b})\neq K_{2^n}$.
So we have $n=s=d+1$. On the other hand if $K(\sqrt{b})\subseteq K_{2^{\infty}}$ we are in the special case described in \eqref{exceptional1} and $n=s=d+1$ implies that $K_{2^n}(\sqrt[2^n]{a})=K_{2^n}(\sqrt{b})=K_{2^{n+1}}$.
\end{proof}

\section{The density of reductions such that the multiplicative order has a prescribed $\ell$-adic valuation }\label{formula1}

From now on let $K$ be a number field and let $a\in K^{\times}$. Let $n\geq 0$. We want to evaluate the  density of the set of primes $\p$ of $K$ such that the multiplicative order of the reduction of $a$ modulo $\p$ has $\ell$-adic valuation $n$, namely
$$D_{\ell}(a,n)=\dens\{\p: \ord_\ell(a \bmod \p)=n \}\,.$$
We always tacitly assume that the reduction $(a \bmod \p)$ and its multiplicative order are well-defined by excluding finitely many primes $\p$.
Note that the set of primes considered for $D_{\ell}(a,n)$ has a natural density and hence also a Dirichlet density (see \cite{Perucca} and remark~\ref{a>0}).

We may suppose that $a$ is not a root of unity, because otherwise the density is trivially either $1$ or $0$, according to whether or not the order of $a$ in $K^\times$ has $\ell$-adic valuation $n$.
Moreover since $K$ is a number field then $a\in K^{\times}$ is not divisible unless it is a root of unity of order coprime to $\ell$. Indeed, this follows from combining two facts: the units in the ring of integers of $K$ form a finitely generated abelian group whose torsion subgroup consists of roots of unity; if $a=b^{\ell^t}$ with $b\in K^\times$ then the exponents in the prime factorization of the fractional ideal generated by $a$
must be all divisible by $\ell^t$, so if $a$ is divisible all exponents are zero and $a$ is a unit. In particular we may assume that $a\xi$ is not divisible in $K$, for every root of unity $\xi\in K$.

Because of the following remark it is sufficient to calculate $D_{\ell}(a,0)$, which we will denote by $D_{\ell}(a)$:

\begin{rem}\label{a>0}
For every $n\geq 1$, we have
$D_{\ell}(a,n)= D_{\ell}(a^{\ell^n},0)-D_{\ell}(a^{\ell^{n-1}},0)\,.$
For $\ell=2$, we have $D_2(a,1)=D_2(-a,0)$ and $D_2(a,n)=D_2(-a, n)$ for every $n\geq 2$.
\end{rem}
\begin{proof}
For the $\ell$-adic valuation of the order we have  
$$\ord_{\ell}(a^{\ell^n} \bmod \p)=\max\big(0, \ord_{\ell}(a \bmod \p)-n\big)$$ and the first assertion follows. Recall that $(-1 \bmod \p)$ is the only element of order $2$ in the multiplicative group of the residue field at $\p$. Then we have either $\ord_{2}(-a \bmod \p)=\ord_{2}(a \bmod \p)\geq 2$ or $\{\ord_{2}(a \bmod \p), \ord_{2}(-a \bmod \p)\}=\{0,1\}$ so  the second assertion follows.
\end{proof}

The value of $D_{\ell}(a)$ can be expressed with the degrees of certain cyclotomic and cyclic Kummer extensions of $K$ (in the same way as for $K=\Q$):

\begin{lem}\label{sum} We have
\begin{equation}\label{summ}
D_{\ell}(a)=\sum_{i\geq 0}\;\Big(\; \frac{1}{[K_{\ell^i}(\sqrt[\ell^{i}]{a}):K]}-\frac{1}{[K_{\ell^{i+1}}(\sqrt[\ell^{i}]{a}):K]}\; \Big)\,.
\end{equation}
\end{lem}
\begin{proof}
Recall that there are only finitely many primes of $K$ that ramify in the field $\cup_{i\geq 0}\, K_{\ell^i}(\sqrt[\ell^{i}]{a})$.
In what follows, we tacitly exclude the finitely many primes $\p$ of $K$ that ramify and those for which the reduction $(a \bmod \p)$ or its multiplicative order is not well-defined.
Let $S$ be the set of primes $\p$ of $K$ such that $(a \bmod \p)$ has order coprime to $\ell$, or equivalently such that $(a \bmod \p)$ has some $\ell^n$-th root in the residue field $k_\p$ for every $n\geq 1$. 
Then we can write $S$ as a disjoint union $S=\cup_{i\geq 0}\, S_i$, where $S_i$  is the subset of $S$ defined by the condition $k_\p\cap \mu_{\ell^\infty}=\mu_{\ell^i}$. The set  $S_i$ then exactly contains the primes which split completely in $K_{\ell^i}(\sqrt[\ell^{i}]{a})$ but do not in $K_{\ell^{i+1}}(\sqrt[\ell^{i}]{a})$. This shows that $S_i$ has a natural density, and we already pointed out that $S$ has a natural density. By the Chebotarev Density Theorem, what we have to prove is the equality
$$\dens(S)=\sum_{i\geq 0} \dens(S_i)\,.$$
For every $n$, we have $\cup_{i\leq n} S_i \subseteq S$ therefore $\dens(S)\geq \sum_{i\leq n} \dens(S_i)$ and in the limit in $n$ we obtain one inequality. 
For the other inequality, remark that $\cup_{i> n} S_i$ is contained in the set of primes of $K$ which split completely in $K_{\ell^{n}}$, and by the Chebotarev Density Theorem this last set has a density that tends to zero as $n$ tends to infinity.

\end{proof}

The rest of this section is devoted to evaluating the density $D_{\ell}(a)$ in the various cases. We use formula \eqref{summ} of lemma~\ref{sum} to express the density. We refer to lemma~\ref{cyclotomic} for the degree of the cyclotomic extension, and we refer to theorems~\ref{odd_or_i}~and~\ref{case_inotK} for the relative degree of the Kummer extension over the cyclotomic extension.

\subsection{The case $\ell$ odd or $\zeta_4\in K$}

\begin{thm}\label{densityodd}
Let $\ell$ be odd, or let $\ell=2$ and $\zeta_4\in K$. Let $t\geq 1$ be the greatest integer such that $K_{\ell}=K_{\ell^t}$. 

\begin{enumerate}
\item If $K\neq K_{\ell}$, then we have

$$D_{\ell}(a)=\left\{ \begin{array}{lll}
\vspace{0.2cm}
1-\frac{1}{[K_{\ell}:K]} \big(1- \frac{\ell}{\ell+1}\cdot\ell^{d-t}\big) & & \text{if}\;\;\; d\leq t \\
1-\frac{1}{[K_{\ell}:K]}\cdot \frac{1}{\ell+1}\cdot \ell^{t-d} & & \text{if}\;\;\; d> t\\
\end{array}
\right.
$$

where $d\geq 0$ is the greatest integer such that $a$ is an ${\ell^d}$-th power in $K$.

\item If $K=K_{\ell}$, then the following holds: 
\begin{enumerate}
\item[(i)] If $a$ is the ${\ell^d}$-th power of a strongly indivisible element, then we have
$$D_{\ell}(a)=\left\{ \begin{array}{lll}
\vspace{0.2cm}
\frac{\ell}{\ell+1}\cdot\ell^{d-t} & & \text{if}\;\;\; d\leq t \\
1-\frac{1}{\ell+1}\cdot \ell^{t-d} & & \text{if}\;\;\; d> t\\
\end{array}
\right.
$$
\item[(ii)] If $a$ is the ${\ell^d}$-th power of a strongly indivisible element times a root of unity of order $\ell^r$ with $r>\max(0,t-d)$, then we have
$$D_{\ell}(a)=\frac{\ell}{\ell+1}\cdot \ell^{-2r+t-d}\,.$$
\end{enumerate}
\end{enumerate}
\end{thm}
\begin{proof}[Proof: Proof of 1.] 
To ease notation, we denote by $\delta$ the inverse of ${[K_{\ell}:K]}$. 
So for every $i\geq t$ the inverse of ${[K_{\ell^{i}}:K]}$ equals ${\delta}\cdot \ell^{t-i}$.
For every $1\leq i\leq t-1$ the $i$-th summand of \eqref{summ} is zero because the two cyclotomic extensions are the same. We also have the equalities ${[K_{\ell^i}(\sqrt[\ell^{i}]{a}):K_{\ell^i}]}={[K_{\ell^{i+1}}(\sqrt[\ell^{i}]{a}):K_{\ell^{i+1}}]}=\ell^{\max(0,i-d)}\,.$ 
So for $d\leq t$ we get the following expression:
$$D_{\ell}(a)=(1-\delta)+ \sum_{i\geq t} \delta (\ell-1)  \ell^{t-i-1}\cdot \ell^{d-i}= 
1-\delta\cdot \Big(1- \frac{\ell}{\ell+1}\cdot \ell^{d-t}\Big)\,.$$
For $d> t$ we have instead the following expression:
$$D_{\ell}(a)=(1-\delta)+ \sum_{i=1}^{d-1}  \delta (\ell-1) \ell^{t-i-1}+ \sum_{i\geq d} \delta (\ell-1) \ell^{t-i-1} \cdot  \ell^{d-i}=1-\delta\cdot  \frac{\ell^{t-d}}{\ell+1}\,.$$
 \emph{Proof of of 2.} The proof of (i) goes exactly as in the previous case by taking $\delta=1$. We are left to prove (ii). For every $i\leq t-1$ the $i$-th summand of \eqref{summ} is zero because the two cyclotomic extensions are the same.  
For every $i\geq t$ we have ${[K_{\ell^{i}}:K]}=\ell^{i-t}$. We also know the following equalities: 
$${[K_{\ell^i}(\sqrt[\ell^i]{a}): K_{\ell^i}]}=\ell^{\max (0, i-d,r)}\qquad\text{and}\qquad {[K_{\ell^{i+1}}(\sqrt[\ell^i]{a}): K_{\ell^{i+1}}]}=\ell^{\max(0, i-d,r-1)}\,.$$
Recall that by assumption we have $t<r+d$. For every $t\leq i\leq r+d-1$ the $i$-th summand of \eqref{summ} is zero because the degrees of the two Kummer extensions are $\ell^r$ and $\ell^{r-1}$ respectively. We may write the remaining terms of \eqref{summ} as follows:
$$D_{\ell}(a)=\sum_{i\geq r+d}\;\;\big( \ell^{t-i}\cdot \ell^{d-i} -\ell^{t-i-1}\cdot \ell^{d-i}\;\big)
=\ell^{-2r+t-d}\cdot \frac{\ell}{\ell+1}\,.$$
\end{proof}

\subsection{The case $\ell=2$ and $\zeta_4\notin K$}
\begin{thm} 
Let $\ell=2$ and suppose that $\zeta_4\notin K$. Let $s\geq 2$ be the greatest integer such that $K_4=K_{2^s}$.
\begin{enumerate}
\item If $a=b^{2^d}$ for some $d\geq 0$ and for some $b\in K^\times$ strongly indivisible, then we have
\begin{equation*}
D_2(a)=\left\{ \begin{array}{lll}
\vspace{0.2cm}
\frac{1}{4}+\frac{\epsilon}{3}\cdot 2^{-s} & &\text{if \; $d=0$}\\
\vspace{0.2cm}
\frac{1}{2}+\frac{\epsilon}{3}\cdot 2^{d-s} & &\text{if \; $0<d<s$}\\
\vspace{0.2cm}
1- \frac{\epsilon}{6}\cdot 2^{s-d} && \text{if \; $d\geq s$}\\
\end{array}
\right.
\end{equation*}
where $\epsilon=\frac{1}{2}$ if $K(\sqrt{b})\subseteq K_{2^{\infty}}$ and $\epsilon=1$ otherwise.

\item If $a=-b^{2^d}$ for some $d>0$ and for some $b\in K^\times$ strongly indivisible, then we have

\begin{equation*} 
D_2(a)=\left\{ \begin{array}{lll}
\vspace{0.2cm}

D_2(-a)-\frac{1}{2}=\frac{\epsilon}{3}\cdot 2^{d-s} & & \text{if \; $0<d<s-1$}\\
\vspace{0.2cm}

D_2(-a)-\frac{\epsilon}{2}=\frac{1}{6\epsilon} & & \text{if \; $d=s-1$}\\
\vspace{0.2cm}

D_2(-a)-1+ \frac{\epsilon}{4}\cdot 2^{s-d}=\frac{\epsilon}{12}\cdot 2^{s-d} & & \text{if \; $d\geq s$}\\
\end{array}
\right.
\end{equation*}
where $\epsilon=\frac{1}{2}$ if $K(\sqrt{b})\subseteq K_{2^{\infty}}$ and $\epsilon=1$ otherwise.
\end{enumerate}
\end{thm}

To determine the value of $\epsilon$ in the above statement we may apply the characterization given in lemma~\ref{squarespecial}.

\begin{proof}[Proof of 1.]
For $i=0$ and for $2\leq i < s$ the $i$-th summand of \eqref{summ} is zero because the two cyclotomic extensions are the same. We may calculate the summand of \eqref{summ} for $i=1$ by considering the equalities $[K(\sqrt{a}): K]=[K_{4}(\sqrt{a}): K_{4}]=2^{\max(0,1-d)}$. For every $i\geq s$ we have ${[K_{2^i}:K]}=2^{1+i-s}$. So we get the following expression:
\begin{multline}\label{last}
D_2(a)=2^{-\max(1,2-d)}+\sum_{i\geq s} 2^{s-i-1} \Big(\frac{1}{[K_{2^i}(\sqrt[2^i]{a}): K_{2^i}]}-\frac{1}{2\cdot [K_{2^{i+1}}(\sqrt[2^{i}]{a}): K_{2^{i+1}}]}\Big)\,.
\end{multline}
\emph{The case $K(\sqrt{b})\not\subseteq K_{2^{\infty}}$.} Formula \eqref{last} may be rewritten  as follows:
$$D_2(a)=2^{-\max(1,2-d)}+\sum_{i\geq s}  2^{s-i-2}\cdot 2^{\min(0, d-i)}\,.$$
If $d=0$ then we have $D_2(a)=2^{-2}+\sum_{i\geq s}  2^{s-2i-2}$ which gives $D_2(a)= \frac{1}{4}+\frac{1}{3}\cdot 2^{-s}$.
If $0<d\leq s$ then we have $D_2(a)=2^{-1}+\sum_{i\geq s}  2^{s+d-2i-2}$ which gives $D_2(a)=\frac{1}{2}+\frac{1}{3}\cdot 2^{d-s}$.
If $d> s$ we have 
$$D_2(a)=2^{-1}+\sum_{i=s}^{d-1} 2^{s-i-2} + \sum_{i\geq d}  2^{s+d-2i-2}=
 1 - \frac{1}{6} \cdot 2^{s-d}\,.$$
  \emph{The case $K(\sqrt{b})\subseteq K_{2^{\infty}}$.} We evaluate formula \eqref{last}.
If $d=0$ we have 
$$D_2(a)=2^{-2}+(2^{-1-s}-2^{-1-s})
+\sum_{i\geq s+1}  \big(2^{s-2i}-2^{s-2i-1}\big) 
=\frac{1}{4} +\frac{1}{6}\cdot 2^{-s}\,.$$
If $0<d<s$ we have
$$D_2(a)=2^{-1}+(2^{-1+d-s} -2^{d-s-1}) +\sum_{i\geq s+1}  \big(2^{s+d-2i}-2^{s+d-2i-1}\big) 
=\frac{1}{2} +\frac{1}{6}\cdot 2^{d-s}\,.$$
If $d=s$ we have $D_2(a)=2^{-1}+2^{-2} +\sum_{i\geq s+1}  \big(2^{2s-2i}-2^{2s-2i-1}\big)= 
\frac{11}{12}\,.$
If $d\geq s+1$ then we have
$$D_2(a)=2^{-1}+2^{-2} +\!\sum_{i=s+1}^{d} \big(2^{s-i-1}-2^{s-i-2}\big)+\!\sum_{i\geq d+1}  \big(\;2^{s+d-2i}-2^{s+d-2i-1}\big) 
=1-\frac{1}{12}\cdot 2^{s-d}\,.
$$
\end{proof}

\begin{proof}[Proof of 2.] We have already evaluated the density for $-a=b^{2^d}$ so it is convenient to consider the difference $\Delta:=D_2(a)-D_2(-a)$. We will then only consider the few summands of \eqref{summ} that are different with respect to those of $-a$. 
Since $d\geq 1$ we have ${[K(\sqrt{-a}): K]}=1$ and ${[K(\sqrt{a}): K]}={2}$. Moreover for every $s\leq i\leq d$ we have ${[K_{2^i}(\sqrt[2^i]{-a}): K_{2^i}]}=1$ and ${[K_{2^i}(\sqrt[2^i]{a}): K_{2^i}]}=2$. 

\noindent \emph{The case $K(\sqrt{b})\not\subseteq K_{2^{\infty}}$.} For $0<d<s$ we have $\Delta=-\frac{1}{2}$, and for $d\geq s$ we have
$$\Delta=-{2^{-1}}+\sum_{i=s}^{d} 2^{s-i-1}(2^{-1}-1)
=-1+ \frac{1}{4}\cdot 2^{s-d}\,.$$
 \emph{The case $K(\sqrt{b})\subseteq K_{2^{\infty}}$.} If $d\geq s-1$ then some more summands of \eqref{summ} have to be considered. If $d\geq s$ then for $i=d+1\geq s+1$ we have ${[K_{2^i}(\sqrt[2^i]{-a}): K_{2^i}]}=1$ and ${[K_{2^i}(\sqrt[2^i]{a}): K_{2^i}]}=2$. So for $d\geq s$ we have 
$$\Delta=\Big(-1+ \frac{1}{4}\cdot 2^{s-d}\Big)+(2^{-1}-1)\cdot 2^{s-d-2}=
-1+ \frac{1}{8}\cdot 2^{s-d}\,.$$
If $d=s-1$ then for $i=s=d+1$ we have ${[K_{2^i}(\sqrt[2^i]{-a}): K_{2^i}]}=2$ and ${[K_{2^i}(\sqrt[2^i]{a}): K_{2^i}]}=1$. So for $d=s-1$ we have $\Delta=\big(-\frac{1}{2}\big)+(1-2^{-1})\cdot 2^{-1}=-\frac{1}{4}\,.$
\end{proof}

In table~\ref{table} we list the values of the density $D_{2}(a,n)$ in the various cases for $K=\Q$. It is convenient to write $a=\pm b^{2^d}$ for some $d>0$ and some rational number $b>0$ which is not a square in $\Q$. The density then depends on the following data: the value of $d$; the sign of $a$; whether $\mathbb Q(\sqrt{b})$ is or is not $\mathbb Q(\sqrt{2})$. The first line of the table shows the generic case. As it may be expected, for $a=b^{2^d}$ the density $D_2(a)$ grows with $d$ and tends to one as $d$ tends to infinity. Consequently for $a=-b^{2^d}$ the density $D_2(a)$ 
tends to zero as $d$ tends to infinity. In any case the density $D_2(a,n)$ tends to zero as $n$ tends to infinity.

We tested our formulas for $D_{\ell}(a)$ with Sage~\cite{sage}: we computed an approximate natural density (considering only the primes with norm $<10^5$) and we found an error smaller than $10^{-2}$ for all examples in tables~\ref{Q3},~\ref{Q2},~\ref{Q22}. 

\begin{table}[ht]\caption{The values of $D_2(a,n)$ for the field $K=\Q$.}
\label{table}
\begin{center}
\begin{tabular}{|l|l||l|l|l|l|}
\hline
  \multirow{2}{*}{ $\mathbb Q(\sqrt{b})$}   & \multirow{2}{*}{$d$}  & $D_2(|a|,0)$ & $D_2(-|a|,0)$ &  \multirow{2}{*}{$D_2(a, 2)$} &  $D_2(a, n)$\\
  & &$=D_2(-|a|,1)$& $=D_2(|a|,1)$ & & for $n\geq 3$\\
 \hline
$\neq \mathbb Q(\sqrt{2})$ & $\geq 0$ & $1-2/3 \cdot 2^{-d}$ & $1/3 \cdot 2^{-d}$ & $1/6\cdot  2^{-d}$ & $2/3\cdot  2^{-d-n}$\\
\cline{1-6}
\multirow{3}{*}{$=\mathbb Q(\sqrt{2})$} & $0$ & $7/24$ & $7/24$ & $1/3$ & \multirow{3}{*}{$1/3\cdot  2^{-d-n}$}\\
\cline{2-5}
 & $1$ & $7/12$ & $1/3$ & \multirow{2}{*}{$1/12\cdot  2^{-d}$} &\\
\cline{2-4}
 & $\geq 2$ & $1-1/3\cdot 2^{-d}$ & $1/6\cdot  2^{-d}$ & &\\
\hline
\end{tabular}
\end{center}
\end{table}

\begin{table}[ht]\caption{Examples for $\ell=3$.}\label{Q3}
\begin{center}
\begin{tabular}{|cc||cc|}
\hline
 $a\in \Q(\sqrt{3})$ & $D_3(a)$ & $a\in \Q(\sqrt{-3})$ & $D_3(a)$\\
\hline
 2  & 5/8 &  2  & 1/4 \\
\hline
 8  & 7/8 &  8  & 3/4 \\
\hline
 $2^9$ & 23/24 & $2^9$ & 11/12  \\
\hline
 3 & 5/8 & $2\cdot \zeta_3$ & 1/4  \\
\hline
 27  & 7/8 &  $8\cdot \zeta_3$  & 1/12  \\
\hline
 2/3 & 5/8 & $2^9\cdot \zeta_3$  & 1/36 \\
\hline
\end{tabular}
\end{center}
\end{table}
\begin{table}[ht]\caption{Examples for  $\ell=2$ and $\zeta_4\in K$.}
\label{Q2}
\begin{center}
\begin{tabular}{|cc|cc|cc|}
\hline
$a\;\in\, \Q(\zeta_4)$ & $D_2(a)$ & $a\;\in \,\Q(\zeta_4)$ & $D_2(a)$ & $a\;\in \,\Q(\zeta_4)$ & $D_2(a)$ \\
\hline
3 & 1/6   &
 -3 & 1/6 &
$\pm 3\zeta_4$ & 1/6 \\ 
\hline
9  & 1/3  &
-9 & 1/3  &
$\pm 9\zeta_4$ & 1/12 \\  
\hline
81   & 2/3   & 
-81  & 1/6     &
$\pm 81\zeta_4$  & 1/24 \\ 
\hline
2    & 1/12 &
 -2   & 1/12 &
 $\pm 2\zeta_4$   & 1/3  \\ 
\hline
4 & 1/6     &
-4  & 2/3   &
$\pm 4\zeta_4$  & 1/24   \\
\hline
16  & 5/6 &
 -16 & 1/12 &
$\pm 16\zeta_4$ & 1/48  \\
\hline
\end{tabular}
\end{center}
\end{table}

\begin{table}[!ht]\caption{Examples for  $\ell=2$ and $\zeta_4\notin K$.}\label{Q22}
\begin{center}
\begin{tabular}{|cc|cc||cc|cc|}
\hline
$a\in \Q(\sqrt{3})$ & $D_2(a)$ & $a\in \Q(\sqrt{3})$ & $D_2(a)$ & $a\in \Q(\sqrt{\pm 2})$ & $D_2(a)$ & $a\in \Q(\sqrt{ \pm2})$ & $D_2(a)$ \\
\hline
3 &  2/3 & 
-3 & 1/6 & 
3 & 7/24    &
 -3 & 7/24  \\ 
\hline
9  & 5/6 &
-9 & 1/12 & 
9  & 7/12  &
-9 & 1/12 \\  
\hline
81   & 11/12   & 
-81  & 1/24   & 
81   & 2/3   & 
-81  & 1/6    \\ 
\hline
2    & 7/24 &
 -2   & 7/24 &
  $\pm 2$   & 7/12 &
 $\mp 2 $    & 1/12\\ 
  \hline
4 & 7/12  &
-4  & 1/3  & 
4 & 2/3   & 
-4  & 1/6  \\ 
\hline
16  & 11/12  & 
 -16 & 1/24 & 
 16  & 5/6   &
 -16 & 1/12 \\ 
\hline
\end{tabular}
\end{center}
\end{table}

\end{document}